\documentclass[11pt]{article}
\usepackage[tbtags]{amsmath}
\usepackage{amssymb}
\usepackage{amsthm}
\usepackage[misc]{ifsym}
\usepackage{cases}
\usepackage{mathrsfs}
\usepackage{ulem}
\usepackage{color}
\usepackage[colorlinks,linkcolor=black,anchorcolor=black,citecolor=black]{hyperref}

\numberwithin{equation}{section}
\normalsize

\title{\bf The Optimal Control Problem of Fully Coupled FBSDEs Driven by Sub-diffusion with Applications \thanks{This work is supported by National Natural Science Foundations of China (12471419, 12271304), Humanities and Social Sciences Foundation of Ministry of Education of China (24YJA910008), Shandong Provincial Natural Science Foundations (ZR2024ZD35, ZR2022JQ01) and Nature Science Foundation of Jiangsu Province (BK20221543).}
}

\author{\normalsize Chenhui Hao\thanks{\it School of Mathematics, Shandong University, Jinan 250100, P.R. China, E-mail: 202411922@mail.sdu.edu.cn},\quad Jingtao Shi\thanks{Corresponding Author. \it School of Mathematics, Shandong University, Jinan 250100, P.R. China, E-mail: shijingtao@sdu.edu.cn},\quad Shuaiqi Zhang\thanks{\it School of Mathematics, China University of Mining and Technology, Xuzhou 221116, P.R. China, E-mail: shuaiqiz@hotmail.com}}

\date{}

\newtheorem{mypro}{Proposition}[section]
\newtheorem{mythm}{Theorem}[section]
\newtheorem{mydef}{Definition}[section]
\newtheorem{mylem}{Lemma}[section]

\newtheorem{assumption}{Assumption}[section]

\begin{document}

\maketitle

\noindent{\bf Abstract:}\quad
This paper is devoted to an optimal control problem of fully coupled forward-backward stochastic differential equations driven by sub-diffusion, whose solutions are not Markov processes. The stochastic maximum principle is obtained, where the control domain may not be convex and the diffusion term is independent of the control variable. Additionally, problem with state constraint is researched by using Ekeland's variational principle. The theoretical results obtained are applied to a cash management optimization problem in bear market, and the optimal strategy is derived.

\vspace{2mm}


\vspace{2mm}

\section{Introduction}

In modern control theory, it is well known that Pontryagin's maximum principle establishes a necessary condition for optimal controls (\cite{BGP1956}). Since Pontryagin's initial proof of the maximum principle in the deterministic situation, researchers have shown sustained interest in its stochastic counterpart. Significant advancements have been achieved in cases where the control domain is convex and the diffusion term is control-dependent, or in cases where the control domain may be non-convex, but the diffusion term is independent of control \cite{Kushner1972,Haussmann1976,Bismut1978,Bensoussan1982}. Notably, Bismut \cite{Bismut1978} introduced linear {\it backward stochastic differential equations} (BSDEs) by applying duality techniques. However, for non-convex control domains where the diffusion term is dependent of control, a general outcome remains unsolved due to the difference in order between It\^o integral $\int_{E_\varepsilon} \sigma dB$ ($\sqrt{\varepsilon}$ order) and Lebesgue integral $\int_{E_\varepsilon} \sigma ds$ ($\varepsilon$ order). In 1990, Peng \cite{Peng1990} used a second-order variation technique, introducing two adjoint equations to derive a general stochastic maximum principle.

In 1990, Pardoux and Peng \cite{PP1990} introduced the following nonlinear BSDE:
$$
-dY(t)=f(t,Y(t),Z(t))dt-Z(t)dB_t,\quad Y(T)=\xi.
$$
Unlike forward {\it stochastic differential equations} (SDEs), BSDEs specify terminal conditions, with solutions in pairs $(Y(\cdot), Z(\cdot))$. \cite{PP1990} proved the existence and uniqueness of solutions under linear growth condition and Lipschitz condition. Since then, mathematicians and economists recognized BSDEs' theoretical and practical significance in fields such as stochastic analysis, stochastic control, and financial mathematics (\cite{DE1992,Peng1993,EPQ2001,CE2002}).

A controlled {\it forward-backward stochastic differential equation} (FBSDE) consists of a forward equation and a backward equation:
\begin{equation*} \left\{\begin{aligned}
dX(t) &= b(t, X(t), Y(t), Z(t), v(t)) dt + \sigma(t, X(t), Y(t), Z(t), v(t)) dB_t,\\
dY(t) &= -f(t, X(t), Y(t), Z(t), v(t)) dt + Z(t) dB_t,\\
X(0) &= x_0, \quad Y(T) = \phi(X(T)),
\end{aligned} \right.\end{equation*}
with solution triple $(X(\cdot),Y(\cdot),Z(\cdot))$, where $v(\cdot)$ is the control. Researchers have extensively studied the stochastic maximum principle for controlled FBSDEs. Peng first derived a stochastic maximum principle for partially coupled FBSDEs under convex control domains \cite{Peng1993}. Xu \cite{Xu1995} extended this result to partially coupled FBSDEs with general control domains provided that the diffusion term does not include the control variable. Wu \cite{Wu2013} considered general maximum principle for partially coupled FBSDEs, by treating $Z$ as one of control variables. Hu \cite{Hu2017} introduced a new variation technique to deal with $z$, to obtained a general stochastic maximum principle with recursive utilities. Hu and Peng \cite{HP1995} and Peng and Wu \cite{PW1999} introduced monotonicity conditions to establish the existence and uniqueness of fully coupled FBSDEs. Based on these results, Wu \cite{Wu1998} demonstrated a stochastic maximum principle for fully coupled FBSDEs under convex control domains. Shi and Wu \cite{SW2006} extended the maximum principle to general control domains where the forward equation's diffusion term is independent of control. Yong \cite{Yong2010} obtained a general optimality variational principle for controlled fully coupled FBSDEs with mixed initial-terminal conditions. Hu, Ji and Xue \cite{HJX2018} identified a new variation term in $Z$, $\langle p(t), \delta\sigma(t) \rangle I_{E_{\varepsilon}}(t)$, enabling the derivation of a general stochastic maximum principle for fully coupled FBSDEs. Scholars have also studied stochastic maximum principles under various conditions, such as state constraints, partial information and partial observation \cite{Shi2007,Meng2009,SW2010,LS2023}.

As Zhang and Chen \cite{ZC2024-SICON-1} pointed out that, compared to Brownian motion, anomalous sub-diffusion more clearly and accurately describes the characteristics of inactive market trading.
Suppose that $S_t$ is a subordinator with drift $\kappa >0$ and L\'evy measure $\nu$; that is, $S_t=\kappa t+ S_t^0$, where $S_t^0$ is a driftless subordinator with L\'evy measure $\nu$.
Let $ L_t:= \inf\{r>0: S_r >t\}$, $t\geq 0$, be the inverse of $S$. The inverse subordinator $L_t$ is continuous in $t$ but stays constant during infinitely many time periods which result from the infinitely many jumps by the subordinator $S_t$ when its L\'evy measure is non-trivial. $B_t$ is the standard Brownian motion on $\mathbb{R}$, and is independent of $B_t$. When the L\'evy measure $\nu$ of  the subordinator $S_t$ is infinite,
$S_t$ has infinitely many small jumps in any given time interval, and for any $\epsilon >0$, the jumps of $S_t$ of size larger than $\epsilon$ occur according to a Poisson process with parameter $\nu (\epsilon, \infty)$.
In this case, during any fixed time intervals, $L_t$ has infinitely many small time periods but only finite many time periods larger than a fixed length  during which it stays constants. Thus the sub-diffusion $B_{L_t}$  matches well the phenomena that the market constantly  has small corrections but big corrections or bear market occurs only sporadically.

Research on the Pontryagin's maximum principle of stochastic optimal problems driven by sub-diffusion remains limited, mainly due to the non-Markovian nature of this process.
Zhang and Chen \cite{ZC2024-SICON-1} established the existence and uniqueness of solutions for BSDEs driven by sub-diffusion. Using convex and spike variation methods, they derived stochastic maximum principles for both convex and non-convex control domains and sufficient ones for general and convex control domains. Additionally, they get the existence and uniqueness for the solution of fully coupled FBSDEs driven by sub-diffusion and the stochastic maximum principle with convex control domain \cite{ZC2024-JDE,ZC2024-SICON-2}.

In this paper, we study the general stochastic maximum principle for fully coupled FBSDEs driven by sub-diffusion for which the diffusion term is independent of control.
In Section 2, we begin by outlining the fundamental properties of the sub-diffusion and proceed to introduce the stochastic optimal control problem driven by sub-diffusion. In Section 3, we obtain the stochastic maximum principle for FBSDEs driven by sub-diffusion, which is one of the main results in this paper. In Section 4, we consider the state constraints and obtain the stochastic maximum principle by Ekeland's variational principle. In Section 5, theoretical results are used to model cash management optimization in bear markets and the optimal strategy are obtained. Some concluding remarks are given in Section 6.

\section{Problem statement and preliminaries}

\subsection{Properties of sub-diffusion}

Let $0<T<\infty$ be a constant time duration. Let $(\Omega,\mathcal{F},\mathbb{P})$ be a given complete probability space, on which a standard one-dimensional Brownian motion $\{B_t\}_{t\geqslant 0}$ is defined. $\mathbb{E}$ denotes the expectation with respect to $\mathbb{P}$.
\begin{mylem}
Suppose $S_t$ is a subordinator and the L\'{e}vy triple of $S_t$ is $(\kappa,0,v)$, let $L_t=\inf\{r>0\colon S_r>t\}$. By definition $ 0\leqslant L_{t+s}-L_{t}\leqslant s/\kappa$, and $dL_t$ is absolutely continuous to Lebesgue measure $dt$. Let $s\rightarrow0$, then:
\begin{equation}\label{subdiffusion}
	\frac{dL_t}{dt}\quad\text{exists with}\quad0\le\frac{dL_t}{dt}\leqslant1/\kappa,\quad\text{for} \quad a.e.\ t>0.
\end{equation}
\end{mylem}

Obviously, $B_{L_t}$ is a continuous martingale with $\left\langle B_{L_t} \right\rangle=L_t$. A significant difference between sub-diffusion and Brownian motion is that it is not a Markov process, so traditional methods do not apply. Zhang and Chen \cite{ZC2024-SICON-1} addressed this issue by introducing {\it overshot process} that imparts Markovian properties to the modified process.

\begin{mylem}
\begin{equation}\label{overshot}
\widetilde{X}_t:=(X_t,R_t):=\left(x_0+B_{L_{(t-R_0)^+}},R_0+S_{L_{(t-R_0)^+}}-t\right)
\end{equation}
with $\widetilde{X_0}=(x_0,R_0)\in\mathbb{R}\times[0,\infty)$ is a time-homogeneous Markov process taking values in $\mathbb{R}\times[0,\infty).$
\end{mylem}

Let $ R_0=a$, and let the filtrations generated by $X_t$ and $\widetilde{X}_t$ be $\mathcal{F}_t^\prime$ and $\mathcal{F}_t$, respectively. Clearly, $\mathcal{F}_t$ and $\mathcal{F}_t^\prime$ depend on $a$, and $\mathcal{F}_t^\prime\subset\mathcal{F}_t$. We also have the following result.

\begin{mylem}
	For ant $t\geqslant 0$ and $R_0\geqslant0$,
\begin{equation}
\lim\limits_{s\to0+}\frac{1}{s}\mathbb{E}\left[L_{(t+s-R_{0})^{+}}-L_{(t-R_{0})^{+}}\big|\mathcal{F}_{t}\right]=\kappa^{-1}1_{\{R_{t}=0\}}=\frac{dL_{(t-R_{0})^{+}}}{dt}.
\end{equation}
\end{mylem}

\subsection{Existence and uniqueness for solutions of FBSDEs driven by sub-diffusion}

Considering the following fully coupled FBSDE driven by sub-diffusion:
\begin{equation}\label{FBSDE driven by sub-diffusion}
\left\{
\begin{aligned}
	dx(t)&=b(t,x(t),y(t))dt+\delta(t,x(t),y(t),z(t))d{L_{(t-a)+}}+\sigma(t,x(t),y(t),z(t))dB_{L_{(t-a)+}},\\
   -dy(t)&=f(t,x(t),y(t))dt+h(t,x(t),y(t),z(t))d{L_{(t-a)+}}-z(t)dB_{L_{(t-a)+}},\\
     x(0)&=x_0,\quad y(T)=\phi(x(T)),
\end{aligned}
\right.
\end{equation}
whose solution triple is $(x(\cdot),y(\cdot),z(\cdot))\in\mathbb{R}\times\mathbb{R}\times\mathbb{R}$. Without loss of generality, let $0\leqslant a < T$. $b, \delta, \sigma, f$ and $h$ are all $\mathcal{F}_t^{\prime}$-adapted stochastic processes. $\phi$ is a square-integrable $\mathcal{F}_T^\prime$-measurable random variable. We need the following hypothesis.

\begin{assumption} \label{assumption}
i) \( b, \sigma, f, \phi, \delta, h \) are all uniformly Lipschitz continuous over their domains, and
\[
\mathbb{E}\left[b^2(0) + \sigma^2(0) + f^2(0) + \delta^2(0) + h^2(0) + \phi^2(0)\right] < \infty.
\]

ii) Monotonicity condition: There exists a constant \( C > 0 \) such that for any \( t > 0 \), \( x_1, x_2, y_1,\) \(y_2, z_1, z_2 \in \mathbb{R} \),
\[
\begin{aligned}
	& \big(b(t, x_1, y_1) - b(t, x_2, y_2)\big)(y_1 - y_2) - \big(f(t, x_1, y_1) - f(t, x_2, y_2)\big)(x_1 - x_2) \\
	& \leqslant -C\big(|x_1 - x_2|^2 + |y_1 - y_2|^2\big),
\end{aligned}
\]
and
\[
\begin{aligned}
	& \big(\sigma(t, x_1, y_1, z_1) - \sigma(t, x_2, y_2, z_2)\big)(z_1 - z_2) + \big(\delta(t, x_1, y_1, z_1) - \delta(t, x_2, y_2, z_2)\big)(y_1 - y_2) \\
	& - \big(h(t, x_1, y_1, z_1) - h(t, x_2, y_2, z_2)\big)(x_1 - x_2) \\
	& \leqslant -C\big(|x_1 - x_2|^2 + |y_1 - y_2|^2 + |z_1 - z_2|^2\big).
\end{aligned}
\]

iii) \( \phi(\cdot) \) is a non-decreasing function.

Conditions ii) and iii) can be replaced by the following:

ii') Monotonicity condition: There exists a constant \( C > 0 \) such that for any \( t > 0 \), \( x_1, x_2, y_1, y_2, z_1, z_2 \in \mathbb{R} \),
\[
\begin{aligned}
& \big(b(t, x_1, y_1) - b(t, x_2, y_2)\big)(y_1 - y_2) - \big(f(t, x_1, y_1) - f(t, x_2, y_2)\big)(x_1 - x_2) \\
& \geqslant C\big(|x_1 - x_2|^2 + |y_1 - y_2|^2\big),
\end{aligned}
\]
and
\[
\begin{aligned}
	& \big(\sigma(t, x_1, y_1, z_1) - \sigma(t, x_2, y_2, z_2)\big)(z_1 - z_2) + \big(\delta(t, x_1, y_1, z_1) - \delta(t, x_2, y_2, z_2)\big)(y_1 - y_2) \\
	& - \big(h(t, x_1, y_1, z_1) - h(t, x_2, y_2, z_2)\big)(x_1 - x_2) \\
	& \geqslant C\big(|x_1 - x_2|^2 + |y_1 - y_2|^2 + |z_1 - z_2|^2\big).
\end{aligned}
\]

iii') \( \phi(\cdot) \) is a non-increasing function.
\end{assumption}

\begin{mydef}
Let
\[
\begin{aligned}
	\mathcal{M}_a[0, T] &:= \Big\{(x(\cdot), y(\cdot), z(\cdot)) : (x(\cdot), y(\cdot), z(\cdot)) \text{ is an adapted process defined} \\
    &\qquad \text{ on } [0, T]\text{ with respect to } \{\mathcal{F}_t^{\prime}\}_{t \in [0, T]} \text{ such that} \\
	&\qquad \mathbb{E}\Big[|x(0)|^2 + \int_0^T(|x(t)|^2 + |y(t)|^2) \, dt + \int_0^T |z(t)|^2 \, dL_{(t-a)+}\Big] < \infty\Big\}.
\end{aligned}
\]
A triple \( (x(\cdot), y(\cdot), z(\cdot)) \in \mathcal{M}_a[0, T] \) is called an adapted solution to (\ref{FBSDE driven by sub-diffusion}) if, for any \( t \in [0, T] \), it satisfies:
\[\left\{
\begin{aligned}
	x(t) &= x_0 + \int_0^t b(s, x(s), y(s)) \, ds + \int_0^t \delta(s, x(s), y(s), z(s)) \, dL_{(s-a)+}\\
         &\quad  + \int_0^t \sigma(s, x(s), y(s), z(s)) \, dB_{L_{(s-a)+}}, \\
	y(t) &= \phi(x(T)) + \int_t^T g(s, x(s), y(s)) \, ds + \int_t^T h(s, x(s), y(s), z(s)) \, dL_{(s-a)+}\\
         &\quad - \int_t^T z(s) \, dB_{L_{(s-a)+}}.
\end{aligned}
\right.\]
The adapted solution to (\ref{FBSDE driven by sub-diffusion}) is unique if for any two adapted solutions \( (x(\cdot), y(\cdot), z(\cdot)) \) and \( (x'(\cdot), y'(\cdot), z'(\cdot)) \), we have \( x(\cdot) = x'(\cdot) \), \( y(\cdot) = y'(\cdot) \), $\mathbb{P}$-a.s., and \( \mathbb{E}\int_0^T |z(s) - z'(s)|^2 \, dL_{(s-a)^+} = 0 \).
\end{mydef}

From \cite{ZC2024-JDE}, we have the following result.

\begin{mythm}\label{th1}
If the FBSDE (\ref{FBSDE driven by sub-diffusion}) satisfies Assumption (\ref{assumption}), then it admits a unique adapted solution $(x(\cdot),y(\cdot),z(\cdot))\in \mathcal{M}_a[0, T]$.
\end{mythm}

\subsection{Stochastic optimal control problem for FBSDEs driven by sub-diffusion}

Let \( U \) be a set on \( \mathbb{R} \). We denote the quintuple \( (\Omega, \mathcal{F}, \mathbb{P}, X, v) \) as \( \mathcal{U}_a[0, T] \) if it satisfies:

i) \( \widetilde{X}_\cdot \) (defined in Lemma \ref{overshot}) is a Markov process with initial value \( \widetilde{X}_0 = (0, a) \).

ii) \( v(t, \omega) \) is a \( \mathcal{F}_t \)-progressively measurable stochastic process and \( \mathbb{E}\int_0^T v(t)^2 \, dt < \infty \). If \( v(t, \omega) \) is a \( \mathcal{F}'_t \)-progressively measurable stochastic process, we denote \( v(\cdot) \in \mathcal{U}_a^\prime[0, T] \). Since \( \mathcal{F}_t^\prime \subset \mathcal{F}_t \), it follows that \( \mathcal{U}_a^\prime[0, T] \subset \mathcal{U}_a[0, T] \). We call \( \mathcal{U}_a^\prime[0, T] \) the admissible control set.

Given \( v(\cdot) \in \mathcal{U}_a^\prime[0, T] \) and \( x_0 \in \mathbb{R} \), consider the controlled FBSDE driven by sub-diffusion:
\begin{equation}\label{controlled FBSDE driven by sub-diffusion}
\left\{
\begin{aligned}
	dx^v(t) &= b(t, x^v(t), y^v(t), v(t)) + \sigma(t, x^v(t), y^v(t), z^v(t)) dB_{L_{(t-a)+}}, \\
	-dy^v(t) &= f(t, x^v(t), y^v(t), v(t)) dt - z^v(t) dB_{L_{(t-a)+}}, \\
	x^v(0) &= x_0, \quad y^v(T) = \phi(x^v(T)),
\end{aligned}
\right.
\end{equation}
where
\[
(x^v(\cdot), y^v(\cdot), z^v(\cdot)) \in \mathbb{R} \times \mathbb{R} \times \mathbb{R}, \quad x_0 \in \mathbb{R},
\]
\[
b: [0, T] \times \mathbb{R} \times \mathbb{R} \times U \to \mathbb{R},
\]
\[
\sigma: [0, T] \times \mathbb{R} \times \mathbb{R} \times \mathbb{R} \to \mathbb{R}, \quad g: [0, T] \times \mathbb{R} \times \mathbb{R} \times U \to \mathbb{R}, \quad \phi: \mathbb{R} \to \mathbb{R}.
\]
The cost functional defined on \( \mathcal{U}_a^\prime[0, T] \) is
\begin{equation}\label{cost functional}
J(v(\cdot)) =\mathbb{E} \left[ \int_0^T g(t, x^v(t), y^v(t), v(t)) dt + h(x^v(T)) + \gamma(y^v(0)) \right],
\end{equation}
where
\[
g: [0, T] \times \mathbb{R} \times \mathbb{R} \times U \to \mathbb{R}, \quad h: \mathbb{R} \to \mathbb{R}, \quad \gamma: \mathbb{R} \to \mathbb{R}.
\]
The goal of the problem is to find an admissible control \( u(\cdot) \in \mathcal{U}_a^{\prime}[0, T] \) such that
\[
J(u(\cdot)) = \inf_{v(\cdot) \in\, \mathcal{U}_a^\prime[0, T]} J(v(\cdot)).
\]

We need the following assumption.
\begin{assumption}\label{assumption2}
i) \( b, \sigma, f, g, h, \gamma, \phi \) are continuously differentiable with respect to their variables, their partial derivatives are bounded and Lipschitz continuous. The minimum values of these Lipschitz constants and bounds are denoted by some constant \( L \).

ii) \( b, \sigma, f \) satisfy the monotonicity condition of Assumption \ref{assumption} ii).

iii) \( \phi \) is a non-decreasing function.
\end{assumption}

From Theorem \ref{th1}, it follows that for all \( v(\cdot) \in \mathcal{U}_a^\prime[0, T] \), the FBSDE (\ref{controlled FBSDE driven by sub-diffusion}) admits a unique adapted solution \( (x^v(\cdot), y^v(\cdot), z^v(\cdot)) \in \mathcal{M}_a[0, T] \).

\section{Stochastic maximum principle for FBSDEs driven by sub-diffusion}

We assume that the control domain $U$ can be non-convex, so $\mathcal{U}_a^\prime[0,T]$ can also be non-convex. To address this, we adopt a spike variation method (\cite{Peng1990}). Let $u(\cdot)$ be an optimal control and $u(\cdot) \in \mathcal{U}_a^\prime[0,T]$, and denote the corresponding adapted solution of (\ref{controlled FBSDE driven by sub-diffusion}) as $(x(\cdot), y(\cdot), z(\cdot))\equiv (x^u(\cdot), y^u(\cdot), z^u(\cdot))$. Make a slight perturbation to $u(\cdot)$ as follows:
$$
u^\varepsilon(t) := {u}(t)1_{E_\varepsilon^c}(t) + v(t)1_{E_\varepsilon}(t),\quad \text{for }t\in[0,T].
$$
Here, $E_\varepsilon$ is a Borel set on $[0,T]$, and the Lebesgue measure $|E_\varepsilon| = \varepsilon$ is sufficiently small. Since $v(\cdot) \in \mathcal{U}_a^\prime[0,T]$, it follows that $u^\varepsilon(\cdot) \in \mathcal{U}_a^\prime[0,T]$. Denote the adapted solution of (\ref{controlled FBSDE driven by sub-diffusion}) corresponding to $u^\varepsilon(\cdot)$ as $(x^\varepsilon(\cdot), y^\varepsilon(\cdot), z^\varepsilon(\cdot))$.

Now we introduce the variational equation:
\begin{equation}\label{variational equation}
	\left\{
	\begin{aligned}
		dx_1(t)&=\left[b_xx_1(t)+b_yy_1(t)+b(u^\varepsilon(t))-b(u(t))\right]dt \\
		&\quad +\left[\sigma_xx_1(t)+\sigma_yy_1(t)+\sigma_zz_1(t)\right]dB_{L_{(t-a)+}}, \\
		-dy_1(t)&=\left[f_xx_1(t)+f_yy_1(t)+f(u^\varepsilon(t))-f(u(t))\right]dt-z_1(t)dB_{L_{(t-a)+}}, \\
		x_1(0)&=0,\quad y_1(T)=\phi_x(x(T))x_1(T),
	\end{aligned}\right.
\end{equation}
where
\begin{equation*}
\begin{aligned}
    b(u^\varepsilon(t))&\equiv b(t, x^\varepsilon(t), y^\varepsilon(t), v^\varepsilon(t)),\quad b(u(t))\equiv b(t,x(t),y(t),u(t)),\\
    b_x&\equiv b_x(t,x(t),y(t),u(t)),\quad b_y\equiv b_y(t,x(t),y(t),u(t)),\quad \text{for }b=b,f,g,
\end{aligned}
\end{equation*}
and
\begin{equation*}
    \sigma_x\equiv \sigma_x(t,x(t),y(t),z(t)),\quad \sigma_y\equiv \sigma_y(t,x(t),y(t),z(t)),\quad \sigma_z\equiv \sigma_z(t,x(t),y(t),z(t)).
\end{equation*}
It is obvious (\ref{variational equation}) satisfies i) and iii) of Assumption \ref{assumption2}. We will introduce a lemma to show that it also satisfies ii).

\begin{mylem}\label{lemma3.1}
Given $t,x,y$, for any $\xi_1,\xi_2 \in \mathbb{R}$, we have
\begin{equation}\label{lemma3.1-1}
	\xi_2\nabla b(t,x,y)\cdot(\xi_1,\xi_2)-\xi_1\nabla f(t,x,y)\cdot(\xi_1,\xi_2) \leqslant -C(\xi_1^2+\xi_2^2).
\end{equation}
Given $t,x,y,z$, for any $\xi_1,\xi_2,\xi_3 \in \mathbb{R}$, we have
\begin{equation}\label{lemma3.1-2}
	\xi_3\nabla\sigma(t,x,y,z)\cdot(\xi_1,\xi_2,\xi_3) \leqslant -C(\xi_1^2+\xi_2^2+\xi_3^2).
\end{equation}
In the above, $\nabla$ denote the gradient operator of a multivariate function, and $\cdot$ denote the dot product of two vectors.
\end{mylem}
\begin{proof}
	Given $t,x,y$, for any $\varepsilon>0$, by the monotonicity condition we have
	$$
\begin{aligned}
	&\big(b(t,x+\xi_1\varepsilon,y+\xi_2\varepsilon)-b(t,x,y)\big)\xi_2\varepsilon -\big(g(t,x+\xi_1\varepsilon,y+\xi_2\varepsilon)-g(t,x,y)\big)\xi_1\varepsilon \\
    &\leqslant -C\varepsilon^2\bigl(\xi_1^2+\xi_2^2\bigr).
	\end{aligned}
    $$
	That is,
	$$
	\begin{aligned}
		&\big(b(t,x+\xi_1\varepsilon,y+\xi_2\varepsilon)-b(t,x,y+\xi_2\varepsilon)+b(t,x,y+\xi_2\varepsilon)-b(t,x,y)\big)\xi_2\varepsilon\\
		&-\big(g(t,x+\xi_1\varepsilon,y+\xi_2\varepsilon)-g(t,x+\xi_1\varepsilon,y)+g(t,x+\xi_1\varepsilon,y)-g(t,x,y)\big)\xi_1\varepsilon\\
		&\leqslant -C\varepsilon^2\bigl(\xi_1^2+\xi_2^2\bigr).
	\end{aligned}
	$$
	Taking $\varepsilon\rightarrow 0$, we can prove (\ref{lemma3.1-1}). Similarly, (\ref{lemma3.1-2}) can be proved.
\end{proof}

Thus, the variational equation satisfies Assumption (\ref{assumption2}), and there exists a unique adapted solution $(x_1(\cdot),y_1(\cdot),z_1(\cdot))\in \mathcal{M}_a[0, T]$. First, we have the following two lemmas.

\begin{mylem}\label{lemma3.2}
Let Assumption \ref{assumption2} hold. Then the following estimations hold for some positive constant $C$:
$$
\int_0^T \mathbb{E}| x_1(t)|^2 dt \leqslant C\varepsilon,\quad \int_0^T \mathbb{E}| y_1(t)|^2 dt \leqslant C\varepsilon,\quad \int_0^T \mathbb{E}| z_1(t)|^2 dL_{(t-a)+} \leqslant C\varepsilon.
$$
\begin{proof}
    The variational equation (\ref{variational equation}) satisfies the monotonicity condition of Assumption (\ref{assumption}), so by Ito's formula we have (Some time variables are omitted when no ambiguity occurs.)
	$$
	\begin{aligned}
		x_1(T)y_1(T)=& \int_0^T y_1 \Big\{\left[b_xx_1+b_yy_1+b(u^\varepsilon)-b(u)\right]dt +\left(\sigma_xx_1+\sigma_yy_1+\sigma_zz_1\right)dB_{L_{(t-a)+}}\Big\}\\
        & -\int_0^T x_1 \Big\{\left[f_xx_1+f_yy_1+f(u^\varepsilon)-f(u)\right]dt-z_1dB_{L_{(t-a)+}}\Big\}\\
		& +\int_0^T z_1\left(\sigma_xx_1+\sigma_yy_1+\sigma_zz_1\right)dL_{(t-a)+}\\
		=&\int_0^T \left[y_1\left(b_xx_1+b_yy_1\right)-x_1\left(f_xx_1+f_yy_1\right)\right]dt\\
		&+\int_0^T \left[y_1\left(b(u^\varepsilon)-b(u)\right)-x_1\left(f(u^\varepsilon)-f(u)\right)\right]dt\\
		&+\int_0^T z_1\left[{\sigma}_xx_1+\sigma_yy_1+\sigma_zz_1\right]dL_{(t-a)+}\\
		&+\int_0^T \left[y_1\left(\sigma_xx_1+\sigma_yy_1+\sigma_zz_1\right)+x_1\left(f_xx_1+f_yy_1\right)\right]dB_{L_{(t-a)+}}\\
		\leqslant& -C\int_0^T \left[x_1^2 dt + y_1^2 dt + z_1^2 dL_{(t-a)+}\right]\\
		&+\int_0^T \left[y_1\left(b(u^\varepsilon)-b(u)\right)-x_1\left(f(u^\varepsilon)-f(u)\right)\right]dt.
	\end{aligned}
	$$
	Taking expectations on both sides, considering the martingale property of sub-diffusion, we obtain
	\begin{equation}\label{important inequality}
		\begin{aligned}
			&C\mathbb{E}\int_0^T \Big[\left(x_1^2 + y_1^2\right)dt + z_1dL_{(t-a)+}\Big] + \mathbb{E}\left[\phi_x\left(x(T)\right) x_1(T)^2\right]\\
			&\leqslant \mathbb{E}\int_0^T \left[y_1\left(b(u^\varepsilon)-b(u)\right)-x_1\left(f(u^\varepsilon)-f(u)\right)\right]dt \\
			&\leqslant \mathbb{E}\int_0^T \left[\frac{1}{C} \left(b(u^{\varepsilon}) - b(u)\right)^2 + \frac{C}{4} y_1^2\right] dt
             + \mathbb{E}\int_0^T \left[\frac{1}{C} \left(f(u^{\varepsilon}) - f(u)\right)^2 + \frac{C}{4} x_1^2\right] dt.
		\end{aligned}
	\end{equation}
	Therefore, there exists a constant $C \geqslant 0$ such that
	\begin{equation*}
		\mathbb{E}\int_0^T \left(x_1(t)^2 + y_1(t)^2\right) dt + \mathbb{E}\int_0^Tz_1(t)^2 dL_{(t-a)+}
        + \mathbb{E}\left(\phi_{x}\left(x(T)\right) x_1(T)^2\right) \leqslant C\varepsilon.
	\end{equation*}
    The proof is complete.
	\end{proof}
\end{mylem}

\begin{mylem}\label{lemma3.3}
Let Assumption \ref{assumption2} hold. Then the following estimations hold for some positive constant $C$:
\begin{equation*}
\begin{gathered}
	\sup_{0 \leqslant t \leqslant T} \mathbb{E}|x_1(t)|^2 \leqslant C\varepsilon, \quad \sup_{0 \leqslant t \leqslant T} \mathbb{E}|y_1(t)|^2 \leqslant C\varepsilon, \\
	\mathbb{E}\left(\sup_{0 \leqslant t \leqslant T} \left|x_1(t)\right|^2\right) \leqslant C\varepsilon, \quad \mathbb{E}\left(\sup_{0 \leqslant t \leqslant T} \left|y_1(t)\right|^2\right) \leqslant C\varepsilon.
\end{gathered}
\end{equation*}
\begin{proof}
    Integrating the BSDE in (\ref{variational equation}) over the interval $[t, T]$, we get
    \begin{equation}\label{BSDE}
    \begin{aligned}
	    y_1(t) + \int_t^T z_1(s) dB_{L_{(s-a)+}} = \phi_x(x(T)) x_1(T) + \int_t^T \left(f_x x_1 + f_y y_1 + f(u^\epsilon) - f(u)\right) ds.
    \end{aligned}
    \end{equation}
    Noting that
    \[
    \mathbb{E}\left[y_1(t) \int_t^T z_1(s) dB_{L_{(s-a)+}}\right] = 0,\quad \text{for any }t\in[0,T].
    \]
    Taking expectations after squaring both sides of (\ref{BSDE}) and applying Lemma \ref{lemma3.2}, by It\^o's isometry, there exists some constant $C_1 \geqslant 0$ such that:
    \begin{equation*}
	\begin{aligned}
		&\mathbb{E}|y_1(t)|^2 + \mathbb{E}\int_t^T \left(z_1(s)\right)^2 dL_{(s-a)+} \\
        &= \mathbb{E}\bigg[\phi_x(x(T)) x_1(T)+ \int_t^T \left(f_x x_1 + f_y y_1 + f(u^\epsilon) - f(u)\right) ds\bigg]^2 \\
		&\leqslant 4\mathbb{E}\left[\left(\phi_x(x(T)) x_1(T)\right)^2 + \int_t^T \Big[(L x_1)^2 + (L y_1)^2 + (f(u^\epsilon) - f(u))^2\Big] ds\right] \\
		&\leqslant C_1 \mathbb{E}\left[x_1^2(T) + \int_t^T \Big[x_1^2(s) + y_1^2(s) + \left(f(u^\epsilon) - f(u)\right)^2\Big] ds\right] \leqslant C\varepsilon.
	\end{aligned}
    \end{equation*}
    Thus,
    \[
    \sup_{0 \leqslant t \leqslant T} \mathbb{E}|y_1(t)|^2 \leqslant C\varepsilon.
    \]
    Integrating the forward SDE in (\ref{variational equation}) over the interval $[0, t]$ using a similar method, we could obtain
    \[
    \sup_{0 \leqslant t \leqslant T} \mathbb{E}|x_1(t)|^2 \leqslant C\varepsilon.
    \]

    By (\ref{BSDE}), since
    \begin{equation*}
	y_1(t) = \phi_x(x(T)) x_1(T) + \int_t^T \left(f_x x_1 + f_y y_1 + f(u^\epsilon) - f(u)\right) ds - \int_t^T z_1 dB_{L_{(s-a)+}},
    \end{equation*}
    using Cauchy-Schwarz's inequality, there exists $C_2 \geqslant 0$ such that
    \begin{equation*}
	\begin{aligned}
		y_1(t)^2 &\leqslant 4\left(\phi_x(x(T)) x_1(T)\right)^2 + 4\left(\int_t^T \left(f_x x_1 + f_y y_1 + f(u^\epsilon) - f(u)\right) ds\right)^2 \\
		&\quad + 4\left(\int_0^t z_1 dB_{L_{(s-a)+}}\right)^2 + 4\left(\int_0^T z_1 dB_{L_{(s-a)+}}\right)^2 \\
		&\leqslant 4\left(\phi_x(x(T)) x_1(T)\right)^2 + 4(T - t) \int_t^T \left(f_x x_1 + f_y y_1 + f(u^\epsilon) - f(u)\right)^2 ds \\
		&\quad + 4\left(\int_0^t z_1 dB_{L_{(s-a)+}}\right)^2 + 4\left(\int_0^T z_1 dB_{L_{(s-a)+}}\right)^2 \\
		&\leqslant 4\left(\phi_x(x(T)) x_1(T)\right)^2 + C_2 \int_0^T \left(x_1^2 + y_1^2 + \left(f(u^\epsilon) - f(u)\right)^2\right) ds \\
		&\quad + 4\left(\int_0^t z_1 dB_{L_{(s-a)+}}\right)^2 + 4\left(\int_0^T z_1 dB_{L_{(s-a)+}}\right)^2.
	\end{aligned}
    \end{equation*}
    Then
    \begin{equation*}
	\begin{aligned}
		\sup_{0 \leqslant t \leqslant T} y_1(t)^2 &\leqslant 4\left(\phi_x(x(T)) x_1(T)\right)^2 + C_2 \int_0^T \left(x_1^2 + y_1^2 + \left(f(u^\epsilon) - f(u)\right)^2\right) ds \\
		&\quad + 4\sup_{0 \leqslant t \leqslant T} \left(\int_0^t z_1 dB_{L_{(s-a)+}}\right)^2 + 4\left(\int_0^T z_1 dB_{L_{(s-a)+}}\right)^2.
	\end{aligned}
    \end{equation*}
    Taking expectations on both sides and applying Burkholder-Davis-Gundy's inequality, there exists $C_3 \geqslant 0$ such that
    \begin{equation*}
	\begin{aligned}
		\mathbb{E}\left(\sup_{0 \leqslant t \leqslant T} y_1(t)^2\right)
        &\leqslant 4\mathbb{E}\left(\phi_x(x(T)) x_1(T)\right)^2 + C_2 \mathbb{E}\int_0^T \left(x_1^2 + y_1^2 + \left(f(u^\epsilon) - f(u)\right)^2\right) ds \\
		&\quad + 4\mathbb{E}\left(\sup_{0 \leqslant t \leqslant T} \int_0^t z_1 dB_{L_{(s-a)+}}\right)^2 + 4\mathbb{E}\left(\int_0^T z_1 dB_{L_{(s-a)+}}\right)^2 \\
		&\leqslant 4\mathbb{E}\left(\phi_x(x(T)) x_1(T)\right)^2 + C_2 \mathbb{E}\left(\int_0^T \left(x_1^2 + y_1^2 + \left(f(u^\epsilon) - f(u)\right)^2\right) ds\right) \\
		&\quad + C_3 \mathbb{E}\left(\int_0^t z_1^2 dL_{(s-a)+}\right) + 4\mathbb{E}\left(\int_0^T z_1^2 dL_{(s-a)+}\right).
	\end{aligned}
    \end{equation*}
    By Lemma \ref{lemma3.2}, there exists $C \geqslant 0$ such that
    \[
    \mathbb{E}\left(\sup_{0 \leqslant t \leqslant T} y_1(t)^2\right) \leqslant C\varepsilon.
    \]
    Using a similar method, we obtain
    \[
    \mathbb{E}\left(\sup_{0 \leqslant t \leqslant T} x_1(t)^2\right) \leqslant C\varepsilon.
    \]
    The proof is complete.
\end{proof}
\end{mylem}

However, we also need the following more elaborate estimations, than those in Lemma \ref{lemma3.2} and Lemma \ref{lemma3.3}.
\begin{mylem}\label{lemma3.4}
Let Assumption \ref{assumption2} hold. Then the following estimations hold for some positive constant $C$:
$$
\begin{gathered}
	\mathbb{E}\int_0^T\left|x_1(t)\right|^2dt\leqslant C\varepsilon^{\frac32},\quad \mathbb{E}\int_0^T\left|y_1(t)\right|^2dt\leqslant C\varepsilon^{\frac32},\quad
	\mathbb{E}\int_0^T\left|z_1(t)\right|^2dL_{(t-a)+}\leqslant C\varepsilon^{\frac32},\\
    \sup_{0\leqslant t\leqslant T}\left(\mathbb{E}\left|x_1(t)\right|^2\right)\leqslant C\varepsilon^{\frac32},\quad
	\sup_{0\leqslant t\leqslant T}\left(\mathbb{E}\left|y_1(t)\right|^2\right)\leqslant C\varepsilon^{\frac32},\\
	\sup_{0\leqslant t\leqslant T}\left(\mathbb{E}\left|x_1(t)\right|^4\right)\leqslant C\varepsilon^3, \quad
	\sup_{0\leqslant t\leqslant T}\left(\mathbb{E}\left|y_1(t)\right|^4\right)\leqslant C\varepsilon^3.
\end{gathered}
$$	
\begin{proof}
    By the first inequality of (\ref{important inequality}), we have
	$$
		\begin{aligned}
			&C\mathbb{E}\int_0^T \left[\left(x_1^2 + y_1^2\right) dt + z_1^2 dL_{(t-a)+}\right] + \mathbb{E}\left[\phi_x(x(T)) x_1(T)^2\right] \\
			&\leqslant \mathbb{E}\int_0^T \left[y_1\left(b(u^\varepsilon) - b(u)\right) - x_1\left(f(u^\varepsilon) - f(u)\right)\right] dt\\
			&\leqslant \mathbb{E}\left[\sup_{0 \leqslant t \leqslant T} \left|y_1(t)\right| \int_0^T \left|b(u^\varepsilon) - b(u)\right| dt\right]
              + \mathbb{E}\left[\sup_{0 \leqslant t \leqslant T} \left|x_1(t)\right| \int_0^T \left|f(u^\varepsilon) - f(u)\right| dt\right] \\
			&\leqslant \left[\mathbb{E}\left(\sup_{0 \leqslant t \leqslant T} \left|y_1(t)\right|^2\right)\right]^{\frac{1}{2}}
              \left[\mathbb{E}\left(\int_0^T \left|b(u^\varepsilon) - b(u)\right| dt\right)^2\right]^{\frac{1}{2}} \\
			&\quad + \left[\mathbb{E}\left(\sup_{0 \leqslant t \leqslant T} \left|x_1(t)\right|^2\right)\right]^{\frac{1}{2}}
              \left[\mathbb{E}\left(\int_0^T \left|f(u^\varepsilon) - f(u)\right| dt\right)^2\right]^{\frac{1}{2}}\\
			&\leqslant C\varepsilon^{\frac{3}{2}}.
		\end{aligned}
	$$
    So that
	$$
		\mathbb{E}\int_0^T\left|x_1(t)\right|^2dt\leqslant C\varepsilon^{\frac32},\quad \mathbb{E}\int_0^T\left|y_1(t)\right|^2dt\leqslant C\varepsilon^{\frac32},\quad
		\mathbb{E}\int_0^T\left|z_1(t)\right|^2dL_{(t-a)+}\leqslant C\varepsilon^{\frac32}.
	$$
    By the BSDE in (\ref{variational equation}), we get
	$$
	\begin{aligned}
		&\mathbb{E}|y_1(t)|^2+\mathbb{E}\int_t^T\left(z_1(s)\right)^2dL_{(s-a)+}\\
	    &=\mathbb{E}\left(\phi_x(x(T))x_1(T)+\int_t^T\left(f_xx_1+f_yy_1+f(u^\epsilon)-f(u)\right)ds\right)^2\\
        &\leqslant C\mathbb{E}\left[(x_1(T))^2+\int_t^T\left(x_1^2+y_1^2\right)ds+\int_t^T(f(u^\varepsilon)-f(u))ds\right]^2\leqslant C\varepsilon^\frac32.
	\end{aligned}
	$$
    Hence, we obtain
	$$
    \sup_{0 \leqslant t \leqslant T} \left(\mathbb{E}\left|y_1(t)\right|^2\right) \leqslant C\varepsilon^{\frac{3}{2}}.
	$$
    Similarly, by the in (\ref{variational equation}), we get
	$$
    \sup_{0 \leqslant t \leqslant T} \left(\mathbb{E}\left|x_1(t)\right|^2\right) \leqslant C\varepsilon^{\frac{3}{2}}.
	$$
    The last two estimations could be proved similarly. The proof is complete.
\end{proof}
\end{mylem}

By the preceding lemmas, we have the following result.
\begin{mylem}\label{lemma3.5}
Let Assumption \ref{assumption2} hold. Then the following inequalities hold:
$$
\begin{gathered}
	\sup\limits_{0\leqslant t\leqslant T}\left[\mathbb{E}\left|x^\varepsilon(t)-x(t)-x_1(t)\right|^2\right]\leqslant C_\varepsilon\varepsilon^2,\quad
    \sup\limits_{0\leqslant t\leqslant T}\left[\mathbb{E}\left|y^\varepsilon(t)-y(t)-y_1(t)\right|^2\right]\leqslant C_\varepsilon\varepsilon^2,\\ \mathbb{E}\int_0^T|z^\varepsilon(t)-z(t)-z_1(t)|^2dL_{(t-a)+}\leqslant C_\varepsilon\varepsilon^2,
\end{gathered}
$$
where $\lim\limits_{\varepsilon\to0}C_\varepsilon=0$.
\begin{proof}
	Define
	$$
		\widetilde{x}(t) \equiv x^\varepsilon(t) - x(t) - x_1(t), \quad \widetilde{y}_t \equiv y^\varepsilon(t) - y(t) - y_1(t), \quad
        \widetilde{z}(t) \equiv z^\varepsilon(t) - z(t) - z_1(t).
	$$
	Notice that
	$$
	\begin{aligned}
		&\int_0^t b(x + x_1, y + y_1, u^\varepsilon) ds + \int_0^t \sigma(x + x_1, y + y_1, z + z_1) dB_{L_{(s-a)+}} \\
		&= \int_0^t \left[ b(x, y, u^\varepsilon) + x_1 \int_0^1 b_x(x + \lambda x_1, y + \lambda y_1, u^\varepsilon) d\lambda \right.\\
        &\qquad\quad \left. +\ y_1 \int_0^1 b_y(x + \lambda x_1, y + \lambda y_1, u^\varepsilon) d\lambda \right] ds \\
		&\quad + \int_0^t \left[ \sigma(x, y, z) + x_1 \int_0^1 \sigma_x(x + \lambda x_1, y + \lambda y_1, z + \lambda z_1) d\lambda \right.\\
		&\qquad\qquad +\ y_1 \int_0^1 \sigma_y(x + \lambda x_1, y + \lambda y_1, z + \lambda z_1) \, d\lambda \\
        &\qquad\qquad \left.+\ z_1 \int_0^1 \sigma_{z}(x + \lambda x_1, y + \lambda y_1, z + \lambda z_1) \, d\lambda \right] dB_{L_{(s-a)+}} \\
		&= x(t) + x_1(t) - x_0 + \int_0^t A^\varepsilon(s) ds + \int_0^t B^\varepsilon(s) dB_{L_{(s-a)+}},
	\end{aligned}
	$$
	and similarly,
	$$
	\begin{aligned}
		&\int_t^T -f(x + x_1, y + y_1, u^\varepsilon) ds + \int_t^T (z + z_1) dB_{L_{(s-a)+}} \\
		&= -\int_t^T \left[ f(x, y, u^\varepsilon) + x_1 \int_0^1 f_x(x + \lambda x_1, y + \lambda y_1, u^\varepsilon) d\lambda \right.\\
		&\qquad\qquad \left.+\ y_1 \int_0^1 f_y(x + \lambda x_1, y + \lambda y_1, u^\varepsilon)  d\lambda \right] ds + \int_t^T (z + z_1) dB_{L_{(s-a)+}} \\
		&= \phi(x(T)) + \phi_x(x(T)) x_1(T) - y(t) - y_1(t) - \int_t^T C^\varepsilon(s) ds,
	\end{aligned}
	$$
	where
	$$
	\begin{aligned}
		A^\varepsilon(s) &:= x_1 \int_0^1 \left[ b_x(x + \lambda x_1, y + \lambda y_1, u^\varepsilon) - b_x \right] d\lambda \\
        &\quad + y_1 \int_0^1 \left[ b_y(x + \lambda x_1, y + \lambda y_1, u^\varepsilon) - b_y \right] d\lambda, \\
		B^\varepsilon(s) &:= x_1 \int_0^1 \left[ \sigma_x(x + \lambda x_1, y + \lambda y_1, z + \lambda z_1) - \sigma_x \right] d\lambda \\
		&\quad + y_1 \int_0^1 \left[ \sigma_y(x + \lambda x_1, y + \lambda y_1, z + \lambda z_1) - \sigma_y \right] d\lambda \\
		&\quad + z_1 \int_0^1 \left[ \sigma_z(x + \lambda x_1, y + \lambda y_1, z + \lambda z_1) - \sigma_z \right] d\lambda, \\
		C^\varepsilon(s) &:= x_1 \int_0^1 \left[ f_x(x + \lambda x_1, y + \lambda y_1, u^\varepsilon) - f_x \right] d\lambda \\
        &\quad + y_1 \int_0^1 \left[ f_y(x + \lambda x_1, y + \lambda y_1, u^\varepsilon) - f_y \right] d\lambda.
	\end{aligned}
	$$	
	For notational simplicity, when \(\psi = b(x, y, u), f(x, y, u)\), we define
	\[
	\begin{aligned}
		&\widetilde{\psi}_x := \int_0^1 \psi_x(x + x_1 + \lambda \widetilde{x}, y + y_1 + \lambda \widetilde{y}, u^\varepsilon) d\lambda, \\
		&\widetilde{\psi}_y := \int_0^1 \psi_y(x + x_1 + \lambda \widetilde{x}, y + y_1 + \lambda \widetilde{y}, u^\varepsilon) d\lambda,
	\end{aligned}
	\]
	and
	\[
	\begin{gathered}
		\widetilde{\sigma}_x := \int_0^1 \sigma_x(x + x_1 + \lambda \widetilde{x}, y + y_1 + \lambda \widetilde{y}, z + z_1 + \lambda \widetilde{z}) \, d\lambda, \\
		\widetilde{\sigma}_y := \int_0^1 \sigma_y(x + x_1 + \lambda \widetilde{x}, y + y_1 + \lambda \widetilde{y}, z + z_1 + \lambda \widetilde{z}) \, d\lambda, \\
		\widetilde{\sigma}_z := \int_0^1 \sigma_z(x + x_1 + \lambda \widetilde{x}, y + y_1 + \lambda \widetilde{y}, z + z_1 + \lambda \widetilde{z}) \, d\lambda.
	\end{gathered}
	\]
	By Assumption \ref{assumption2}, Lemma \ref{lemma3.4}, and \(|E_\varepsilon| = \varepsilon\), we get
	\[
	\begin{gathered}
		\mathbb{E}\int_0^t A^\varepsilon(s)^2 ds \leqslant C_\varepsilon \varepsilon^2, \quad
		\mathbb{E}\int_0^t B^\varepsilon(s)^2  dL_{(s-a)+} \leqslant C_\varepsilon \varepsilon^2, \quad
		\mathbb{E}\int_0^t C^\varepsilon(s)^2 ds \leqslant C_\varepsilon \varepsilon^2.
	\end{gathered}
	\]
    Since \(\widetilde{x}(\cdot),\widetilde{y}(\cdot),\widetilde{z}(\cdot)\) satisfy the following FBSDE with sub-diffusion:
    \begin{equation*}
    \left\{
    \begin{aligned}
	d\widetilde{x}(s) &= \big(\widetilde{b_x} \widetilde{x}(s) + \widetilde{b_y} \widetilde{y}(s) + A^\varepsilon(s)\big) ds
     + \big(\widetilde{\sigma}_x \widetilde{x}(s) + \widetilde{\sigma}_y \widetilde{y}(s) + \widetilde{\sigma}_z \widetilde{z}(s) + B^\varepsilon(s)\big) dB_{L_{(s-a)+}}, \\
	-d\widetilde{y}(s) &= \big(\widetilde{f_x} \widetilde{x}(s) + \widetilde{f_y} \widetilde{y}(s) + C^\varepsilon(s)\big) ds - \widetilde{z}(s) dB_{L_{(s-a)+}}, \\
	\widetilde{x}(0) &= 0, \quad \widetilde{y}(T) = \phi(x^\varepsilon(T)) - \phi(x(T) + x_1(T)),
    \end{aligned}
    \right.
    \end{equation*}
    applying It\^o's formula to \(\widetilde{x}(\cdot) \widetilde{y}(\cdot)\) on \([0, T]\) and using the monotonicity condition, we obtain
    \[
    \begin{aligned}
	\mathbb{E}\left[\widetilde{x}(T) \widetilde{y}(T)\right]
    &= \mathbb{E}\int_0^T \widetilde{y} \Big[\big(\widetilde{b_x} \widetilde{x} + \widetilde{b_y} \widetilde{y} + A^\varepsilon(s)\big) ds
    + \big(\widetilde{\sigma}_{x} \widetilde{x} + \widetilde{\sigma}_{y} \widetilde{y} + \widetilde{\sigma}_{z} \widetilde{z} + B^\varepsilon(s)\big) dB_{L_{(s-a)+}}\Big] \\
	&\quad - \mathbb{E}\int_0^T \widetilde{x} \Big[\big(\widetilde{f_x} \widetilde{x} + \widetilde{f_y} \widetilde{y} + C^\varepsilon(s)\big) ds
     - \widetilde{z} dB_{L_{(s-a)+}}\Big] \\
	&\quad + \mathbb{E}\int_0^T \widetilde{z} \big(B^\varepsilon(s) + \widetilde{\sigma}_x \widetilde{x} + \widetilde{\sigma}_y \widetilde{y}
     + \widetilde{\sigma}_z \widetilde{z}\big) dL_{(s-a)+} \\
	&\leqslant -C\mathbb{E}\int_0^T \left(\widetilde{x}^2 + \widetilde{y}^2\right) ds -C \mathbb{E}\int_0^T \widetilde{z}^2 dL_{(s-a)+} \\
	&\quad + \mathbb{E}\int_0^T \left(\widetilde{y} A^\varepsilon(s) - \widetilde{x} C^\varepsilon(s)\right) ds + \int_0^T \widetilde{z} B^\varepsilon(s) dL_{(s-a)+}.
    \end{aligned}
    \]
    Therefore
    \[
    \begin{aligned}
	&\mathbb{E}\left[\widetilde{x}(T) \widetilde{y}(T)\right] + \mathbb{E}\int_0^T c(\widetilde{x}^2 +\widetilde{y}^2) ds + \mathbb{E}\int_0^T c \widetilde{z}^2 dL_{(s-a)+}\\
	&\leqslant \frac{C}{4}\mathbb{E}\int_0^T \widetilde{y}^2 ds + \frac{1}{C}\mathbb{E}\int_0^T A^\varepsilon(s)^2 ds + \frac{C}{4}\mathbb{E}\int_0^T \widetilde{x}^2 ds\\
	&\quad  + \frac{1}{C}\mathbb{E}\int_0^T C^\varepsilon(s)^2 ds  + \frac{C}{4}\mathbb{E}\int_0^T \widetilde{z}^2 dL_{(s-a)+}
     + \frac{1}{C}\mathbb{E}\int_0^T B^\varepsilon(s)^2 dL_{(s-a)+}.
    \end{aligned}
    \]
    Thus
    \[
    \begin{aligned}
	&\mathbb{E}\left[\widetilde{x}(T) \widetilde{y}(T)\right] + \frac{3C}{4}\mathbb{E}\int_0^T (\widetilde{x}^2 + \widetilde{y}^2) ds
     + \frac{3C}{4} \mathbb{E}\int_0^T \widetilde{z}^2 dL_{(s-a)+} \\
	&\leqslant \frac{1}{C}\mathbb{E}\int_0^T A^\varepsilon(s)^2 ds + \frac{1}{C}\mathbb{E}\int_0^T C^\varepsilon(s)^2 ds
     + \frac{1}{C}\mathbb{E}\int_0^T B^\varepsilon(s)^2 dL_{(s-a)+} \leqslant C_\varepsilon \varepsilon^2.
    \end{aligned}
    \]
    Since \(\phi\) is a monotonic function, by Assumption \ref{assumption2}, we have
    \[
    \mathbb{E}\left[\widetilde{x}(T) \widetilde{y}(T)\right] = \mathbb{E}\left[\big(\phi(x_T^\varepsilon) - \phi(x_T + x_1(T))\big)(x_T^\varepsilon - (x_T + x_1(T)))\right] \geqslant 0.
    \]
    Thus, we obtain
    \[
    \begin{gathered}
	\mathbb{E}\int_0^T \widetilde{x}(s)^2 ds \leqslant C_\varepsilon \varepsilon^2, \quad \mathbb{E}\int_0^T \widetilde{y}(s)^2 ds \leqslant C_\varepsilon \varepsilon^2, \quad
    \mathbb{E}\int_0^T \widetilde{z}(s)^2 dL_{(s-a)+} \leqslant C_\varepsilon \varepsilon^2.
    \end{gathered}
    \]
    Applying the method of Lemma \ref{lemma3.3} once more, we obtain the conclusion.
\end{proof}
\end{mylem}

We then have the following variational inequality.
\begin{mypro}\label{proposition3.1}
Let Assumption \ref{assumption2} hold. Then
\begin{equation}\label{variational inequality}
\begin{aligned}
&\mathbb{E}\int_0^T \left[ g_x x_1(t) + g_y y_1(t) + g(u^\varepsilon(t)) - g(u(t)) \right] dt\\
&+ \mathbb{E} \left[ h_x(x(T)) x_1(T) \right] + \mathbb{E} \left[ \gamma_y(y(0)) y_1(0) \right] \geqslant o(\varepsilon).
\end{aligned}
\end{equation}
\begin{proof}
	The fact that \( J(u^\varepsilon(\cdot)) \geqslant J(u(\cdot)) \) implies
	\begin{align*}
		& \mathbb{E}\int_0^T \left[ g(t, x^\varepsilon, y^\varepsilon, u^\varepsilon) - g(t, x + x_1, y + y_1, u^\varepsilon) \right] dt \\
		& + \mathbb{E}\int_0^T \left[ g(t, x + x_1, y + y_1, u^{\varepsilon}) - g(u) \right] dt \\
		& + \mathbb{E} \left[ h(x^\varepsilon(T)) - h(x(T) + x_1(T)) \right] + \mathbb{E}\left[ h(x(T) + x_1(T)) - h(x(T)) \right] \\
		& + \mathbb{E} \left[ \gamma(y^\varepsilon(0)) - \gamma(y(0) + y_1(0)) \right] + \mathbb{E} \left[ \gamma(y(0) + y_1(0)) - \gamma(y(0)) \right] \geqslant 0.
	\end{align*}
	By Assumption \ref{assumption2} and Lemma \ref{lemma3.5}, we get
	\begin{align*}
		& \mathbb{E}\int_{0}^{T} \left[ g(t, x^\varepsilon, y^\varepsilon, u^\varepsilon) - g(t, x + x_1, y + y_1, u^\varepsilon) \right] dt \\
		& + \mathbb{E} \left[ h(x^{\varepsilon}(T)) - h(x(T) + x_1(T)) \right]
		+ \mathbb{E} \left[ \gamma(y^{\varepsilon}(0)) - \gamma(y(0) + y_1(0)) \right] = o(\varepsilon).
	\end{align*}
	Thus, we have
	\begin{align*}
		0 \leqslant &\ \mathbb{E}\int_0^T \left[ g(t, x + x_1, y + y_1, u^\varepsilon) - g(u) \right] dt + \mathbb{E} \left[ h(x(T) + x_1(T)) - h(x(T)) \right] \\
		& + \mathbb{E} \left[ \gamma(y(0) + y_1(0)) - \gamma(y(0)) \right] + o(\varepsilon) \\
		= &\ \mathbb{E}\int_0^T \left[ g(t, x + x_1, y + y_1, u^\varepsilon) - g(t, x + x_1, y + y_1, u) \right] dt \\
		& + \mathbb{E}\int_0^T \left[ g(t, x + x_1, y + y_1, u) - g(u) \right] dt \\
		& + \mathbb{E} \left[ h(x(T) + x_1(T)) - h(x(T)) \right]
		+ \mathbb{E} \left[ \gamma(y(0) + y_1(0)) - \gamma(y(0)) \right] + C \varepsilon^{\frac{3}{2}} \\
		= &\ \mathbb{E}\int_0^T \left[ g_x x_1(t) + g_y y_1(t) + g(u^\varepsilon) - g(u) \right] dt \\
		& + \mathbb{E}\int_0^T \left[ (g_x(u^\varepsilon) - g_x) x_1 + (g_y(u^\varepsilon) - g_y) y_1 \right] dt \\
		& + \mathbb{E} \left[ h_x(x(T)) x_1(T) \right] + \mathbb{E} \left[ \gamma_y(y(0)) y_1(0) \right] + o(\varepsilon),
	\end{align*}
	where we have denoted $g_x(u^\varepsilon) := g_x(t, x, y, u^\varepsilon)$, $g_y(u^\varepsilon) := g_y(t, x, y, u^\varepsilon)$.

	Since
	\begin{align*}
		& \mathbb{E}\int_0^T \left[ (g_x(u^\varepsilon) - g_x) x_1 + (g_y(u^\varepsilon) - g_y) y_1 \right] dt \\
		\leqslant & \left[ \mathbb{E}\left( \int_0^T (g_x(u^\varepsilon) - g_x) dt \right)^2 \right]^{\frac{1}{2}}
        \left[ \mathbb{E}\left( \sup_{0 \leqslant t \leqslant T} \left| x_1(t) \right|^2 \right) \right]^{\frac{1}{2}} \\
		& + \left[ \mathbb{E}\left( \int_0^T (g_y(u^\varepsilon) - g_y) dt \right)^2 \right]^{\frac{1}{2}}
        \left[ \mathbb{E}\left( \sup_{0 \leqslant t \leqslant T} \left| y_1(t) \right|^2 \right) \right]^{\frac{1}{2}} = o(\varepsilon).
	\end{align*}
	we have
	$$
	\mathbb{E}\int_0^T \left[ g_x x_1 + g_y y_1 + g(u^\varepsilon) - g(u) \right] dt
	+ \mathbb{E} \left[ h_x(x(T)) x_1(T) \right]
	+ \mathbb{E} \left[ \gamma_y(y(0)) y_1(0) \right] \geqslant o(\varepsilon).
	$$
	Thus, (\ref{variational inequality}) holds.
\end{proof}
\end{mypro}

To establish the stochastic maximum principle, we introduce the following adjoint equation:
\begin{equation}\label{adjoint equation}
	\left\{
	\begin{aligned}
		dp(t) &= \left[ f_y p(t) - b_y q(t) - g_y \right] dt - \sigma_y k(t) dL_{(t-a)+} - \sigma_z k(t) dB_{L_{(t-a)+}}, \\
		-dq(t) &= \left[ -f_x p(t) + b_x q(t) + g_x \right] dt + \sigma_x k(t) dL_{(t-a)+} - k(t) B_{L_{(t-a)+}}, \\
		p(0) &= -\gamma_y(y(0)), \quad q(T) = -\phi_x(x(T)) p(T) + h_x(x(T)).
	\end{aligned}
	\right.
\end{equation}

From Theorem \ref{th1}, there exists a unique adapted solution \( (p(\cdot), q(\cdot), k(\cdot)) \in \mathcal{M}_a[0, T]\) to (\ref{adjoint equation}). Define the Hamiltonian function as follows:
$$
H(t, x, y, z, v, p, q) := q(t) b(t, x, y, v) - p(t) f(t, x, y, v) + g(t, x, y, v).
$$

The next theorem provides the main result of this section.
\begin{mythm}\label{theorem3.1}
Let \( u(\cdot)\) be an optimal control, \( (x(\cdot), y(\cdot), z(\cdot)) \) the corresponding solution to the FBSDE with sub-diffusion (\ref{controlled FBSDE driven by sub-diffusion}), and \( (p(\cdot), q(\cdot), k(\cdot)) \) the solution to the adjoint equation (\ref{adjoint equation}). Then for any \( v(\cdot) \in \mathcal{U}_a^\prime[0, T] \), we have
\begin{equation}\label{SMP}
H(t, x(t), y(t), z(t), v(t), p(t), q(t)) \geqslant H(t, x(t), y(t), z(t), u(t), p(t), q(t)),\ a.e.\ t\in[0,T],\ a.s..
\end{equation}
\begin{proof}
	By applying Ito's formula to \( p(\cdot)y_1(\cdot) + q(\cdot)x_1(\cdot) \) over the interval \([0, T]\), we get
	\begin{align*}
		& p(T)y_1(T) + q(T)x_1(T) - p(0)y_1(0) - q(0)x_1(0) \\
		&= \int_0^T \Big\{-p\left[f_x x_1 + f_y y_1 + f(u^\varepsilon) - f(u)\right] dt + \left[f_y p - b_y q - g_y\right] y_1 dt \\
		& \qquad\quad - \left[k\sigma_y y_1 + k\sigma_z z_1\right] dL_{(t-a)+}\Big\} \\
		& \quad + \int_0^T \Big\{\left[b_x x_1 + b_y y_1 + b(u^\varepsilon) - b(u)\right] q dt + \left[f_x p - b_x q - g_x\right] x_1 dt \\
		& \qquad\qquad + \left[k\sigma_y y_1 + k\sigma_z z_1\right] dL_{(t-a)+}\Big\} + \text{martingale}.
	\end{align*}
	By substituting the adjoint equation (\ref{adjoint equation}) and taking expectations on both sides:
	\begin{align*}
		&\mathbb{E} \left[ h_x(x(T)) x_1(T) + \gamma_y(y(0)) y_1(0) \right]\\
		&= \mathbb{E} \int_0^T \Big\{-p \left[ f(u^\varepsilon) - f(u) \right] + q \left[ b(u^\varepsilon) - b(u) \right] - (g_x x_1 + g_y y_1)\Big\} dt.
	\end{align*}
	By Proposition \ref{proposition3.1}, we have
	$$
		\mathbb{E} \int_0^T \Big\{-p \left[ f(u^\varepsilon) - f(u) \right] + q \left[ b(u^\varepsilon) - b(u) \right] + g(u^\varepsilon) - g(u)\Big\} dt \geqslant o(\varepsilon).
	$$
	Combining this with the definition of the Hamiltonian, and letting \( \varepsilon \to 0 \), (\ref{SMP}) holds.
\end{proof}
\end{mythm}

\section{Stochastic maximum principle for FBSDEs driven by sub-diffusion with state constraints}

In practical problems, it is common to encounter situations with state constraints. This section is concerned with the stochastic maximum principle for FBSDEs driven by sub-diffusion with state constraints.

For the readers' convenience, we first introduce the following Ekeland's variational principle, which can be seen, for example, in Yong and Zhou \cite{YZ1999}.

\begin{mylem}\label{Ekeland variational principle}
Let \( V \) be a complete metric space, and let \( f \) be a lower semi-continuous function defined on \( V \), i.e., for any \( v \in V \), \( \varliminf\limits_{v \to v_0} f(v) \geqslant f(v_0) \). Suppose \( \inf\limits_{v \in V} f(v) > -\infty \). For any given \( \varepsilon \geqslant 0 \), there exists \( u \in V \) such that $f(u) \leqslant \inf\limits_{v \in V} f(v) + \varepsilon$, then there exists, for any \( \lambda \geqslant 0 \), a point \( u^\varepsilon \in V \) such that:
$$
\begin{aligned}
	&\mathrm{i}) \quad f(u^{\varepsilon}) \leqslant f(u), \\
	&\mathrm{ii}) \quad d(u^{\varepsilon}, u) \leqslant \lambda, \\
	&\mathrm{iii}) \quad f(u^{\varepsilon}) \leqslant f(w) + \frac{\varepsilon}{\lambda} d(w, u^{\varepsilon}), \quad \forall w \in V.
\end{aligned}
$$
\end{mylem}

Now, we introduce a stochastic optimal control problem of FBSDEs with sub-diffusion with state constraints. Considering the following controlled forward-backward stochastic system:
\begin{equation}\label{controlled forward-backward stochastic system}
\left\{
	\begin{aligned}
		dx^v(t) &= b(t, x^v(t), y^v(t), v(t)) dt + \sigma(t, x^v(t), y^v(t), z^v(t)) dB_{L_{(t-a)+}}, \\
		-dy^v(t) &= f(t, x^v(t), y^v(t), v(t)) dt - z^v(t) dB_{L_{(t-a)+}}, \\
		x^v(0) &= x_0, \quad y^v(T) = \phi(x^v(T)),
	\end{aligned}
\right.
\end{equation}
where
$$ (x^v(\cdot), y^v(\cdot), z^v(\cdot)) \in \mathbb{R} \times \mathbb{R} \times \mathbb{R}, \quad x_0 \in \mathbb{R}, $$
$$ b: [0, T] \times \mathbb{R} \times \mathbb{R} \times U \to \mathbb{R}, \quad \sigma: [0, T] \times \mathbb{R} \times \mathbb{R} \times \mathbb{R} \to \mathbb{R}, $$
$$ g: [0, T] \times \mathbb{R} \times \mathbb{R} \times U \to \mathbb{R}, \quad \phi: \mathbb{R} \to \mathbb{R}. $$

Unlike the previous section, here we assume that \( \{ v(t) \mid t \in [0, T], v(t) \in \mathcal{U}_a^\prime[0, T] \} \) is a bounded closed set, and the state variables \( x^v(T) \) and \( y^v(0) \) satisfy the following constraints:
\begin{equation}\label{state constraints}
\mathbb{E}[G_1(x^v(T))] = 0, \quad \mathbb{E}[G_0(y^v(0))] = 0,
\end{equation}
where \( G_1 \) and \( G_0 \) are continuously differentiable functions from \( \mathbb{R} \) to \( \mathbb{R} \) with bounded derivatives.

The cost functional defined on \( \mathcal{U}_a^\prime[0, T] \) is
$$
J(v(\cdot)) = \mathbb{E} \left[ \int_0^T g\big(t, x^v(t), y^v(t), v(t)\big) dt + h\big(x^v(T)\big) + \gamma\big(y^v(0)\big) \right],
$$
where \( g: [0, T] \times \mathbb{R} \times \mathbb{R} \times U \to \mathbb{R} \), \( h: \mathbb{R} \to \mathbb{R} \), and \( \gamma: \mathbb{R} \to \mathbb{R} \).

The goal in this section is to find an admissible control \( u(\cdot) \in \mathcal{U}_a^\prime[0, T] \) such that
$$
J(u(\cdot)) = \inf_{v(\cdot) \in\, \mathcal{U}_a^\prime[0, T]} J(v(\cdot)).
$$
	
In order to apply Lemma \ref{Ekeland variational principle}, we define a metric on the control domain \( \mathcal{U}_a^\prime[0,T] \) as follows:
$$
d(v_1(\cdot), v_2(\cdot)) := \mathbb{E}\left[\mu\{t \in [0, T] \mid v_1(t) \neq v_2(t)\}\right], \quad \forall v_1(\cdot), v_2(\cdot) \in \mathcal{U}_a^\prime[0,T],
$$
where \( \mu \) denotes the Lebesgue measure. It can be shown that under this metric, \( \mathcal{U}_{a}^{\prime}[0,T] \) is complete.

Let \( u(\cdot) \) be an optimal control, and \( (x(\cdot), y(\cdot), z(\cdot))\equiv (x^u(\cdot), y^u(\cdot), z^u(\cdot)) \in \mathcal{M}_a[0, T]\) be the corresponding solution to (\ref{controlled forward-backward stochastic system}). For any \( v(\cdot) \in \mathcal{U}_{a}^{\prime}[0,T] \), define
$$
J_\rho(v(\cdot)) := \left\{|\mathbb{E}[G_1(x^v(T))]|^2 + |\mathbb{E}[G_0(y^v(0))]|^2 + [J(v(\cdot)) - J(u(\cdot)) + \rho]^2\right\}^{\frac{1}{2}}.
$$
Obviously
$$
J_\rho(u(\cdot)) \leqslant \inf_{v(\cdot) \in\, \mathcal{U}_a^\prime[0,T]} J_\rho(v(\cdot)) + \rho.
$$
Thus, by Lemma \ref{Ekeland variational principle}, we obtain the existence of \( u_\rho(\cdot) \in \mathcal{U}_a^\prime[0,T] \) such that:
$$
\begin{aligned}
	&\mathrm{i}) \quad J_\rho(u_\rho(\cdot)) \leqslant J_\rho(u(\cdot)) = \rho, \\
	&\mathrm{ii}) \quad d(u_\rho(\cdot), u(\cdot)) \leqslant \sqrt{\rho}, \\
	&\mathrm{iii}) \quad J_\rho(u_\rho(\cdot)) \leqslant J_\rho(w(\cdot)) + \sqrt{\rho} \, d(u_\rho(\cdot), w(\cdot)), \quad \forall w(\cdot) \in \mathcal{U}_a^\prime[0,T].
\end{aligned}
$$
We denote the solution to (\ref{controlled forward-backward stochastic system}) with respect to $u_\rho(\cdot)$ as $(x_\rho(\cdot), y_\rho(\cdot), z_\rho(\cdot))$.
Considering a spike variation \( u_\rho^\varepsilon(\cdot) \) of \( u_\rho(\cdot) \) defined by
$$
u_\rho^\varepsilon(t) := u_\rho(t) 1_{E_\varepsilon^c}(t) + v(t) 1_{E_\varepsilon}(t),\quad t\in[0,T],
$$
and let \( (x_\rho^\varepsilon(\cdot), y_\rho^\varepsilon(\cdot), z_\rho^\varepsilon(\cdot)) \) be the solution to (\ref{controlled forward-backward stochastic system}) with respect to $u_\rho^\varepsilon(\cdot)$. Then we obtain from the above that
$$
J_\rho(u_\rho^\varepsilon(\cdot)) - J_\rho(u_\rho(\cdot)) + \sqrt{\rho} \varepsilon \geqslant 0.
$$

Now we introduce the variational equation
\begin{equation}\label{variational equation-state constraint}
	\left\{
	\begin{aligned}
		dx_\rho^1(t) &= \left[b_{x,\rho} x_\rho^1(t) + b_{y,\rho} y_\rho^1(t) + b(u_\rho^\varepsilon(t)) - b(u_\rho(t))\right] dt \\
		&\quad + \left[\sigma_{x,\rho} x_\rho^1(t) + \sigma_{y,\rho} y_\rho^1(t) + \sigma_{z,\rho} z_\rho^1(t)\right] dB_{L_{(t-a)+}}, \\
		-dy_\rho^1(t) &= \left[f_{x,\rho} x_\rho^1(t) + f_{y,\rho} y_\rho^1(t) + f(u_\rho^\varepsilon(t)) - f(u_\rho(t))\right] dt - z_\rho^1(t) dB_{L_{(t-a)+}}, \\
		x_\rho^1(0) &= 0, \quad y_\rho^1(t) = \phi_x(x_\rho(T)) x_\rho^1(T),
	\end{aligned}
	\right.
\end{equation}
where
\begin{equation*}
\begin{aligned}
    b(u_\rho^\varepsilon(t))&\equiv b(t, x_\rho^\varepsilon(t), y_\rho^\varepsilon(t), v_\rho^\varepsilon(t)),\quad b(u_\rho(t))\equiv b(t,x_\rho(t),y_\rho(t),u_\rho(t)),\\
    b_{x,\rho}&\equiv b_x(t,x_\rho(t),y_\rho(t),u_\rho(t)),\quad b_{y,\rho}\equiv b_y(t,x_\rho(t),y_\rho(t),u_\rho(t)),\quad \text{for }b=b,f,g,
\end{aligned}
\end{equation*}
and
\begin{equation*}
    \sigma_{x,\rho}\equiv \sigma_x(t,x_\rho(t),y_\rho(t),z_\rho(t)),\quad \sigma_{y,\rho}\equiv \sigma_y(t,x_\rho(t),y_\rho(t),z_\rho(t)),\quad
    \sigma_{z,\rho}\equiv \sigma_z(t,x_\rho(t),y_\rho(t),z_\rho(t)).
\end{equation*}
It can similarly be shown that (\ref{variational equation-state constraint}) satisfies Assumption \ref{assumption2}, hence it has a unique adapted solution $(x_\rho^1(\cdot), y_\rho^1(\cdot), z_\rho^1(\cdot))\in \mathcal{M}_a[0, T]$. Using the same methods from Lemmas \ref{lemma3.2} to \ref{lemma3.5} in Section 3, we can ontain the following result.

\begin{mylem}\label{lemma4.2}
Let Assumption \ref{assumption2} hold. Then the following inequalities hold:
$$
\begin{gathered}
	\sup\limits_{0 \leqslant t \leqslant T} \left[\mathbb{E} \left| x^\varepsilon_\rho(t) - x_\rho(t) - x^1_\rho(t) \right|^2\right] \leqslant C_\varepsilon \varepsilon^2, \\
	\sup\limits_{0 \leqslant t \leqslant T} \left[\mathbb{E} \left| y^\varepsilon_\rho(t) - y_\rho(t) - y^1_\rho(t) \right|^2\right] \leqslant C_\varepsilon \varepsilon^2, \\
	\mathbb{E} \int_0^T |z^\varepsilon_\rho(t) - z_\rho(t) - z_\rho^1(t)|^2 dL_{(t-a)+} \leqslant C_\varepsilon \varepsilon^2,
\end{gathered}
$$
where \( \lim\limits_{\varepsilon \to 0} C_\varepsilon = 0 \).
\end{mylem}

From this result, we can derive the following variational inequality.
\begin{mypro}\label{proposition4.1}
Let Assumption \ref{assumption2} hold. Then there exist non-zero constants \( \psi_1 \), \( \psi_2 \), and \( \psi_3 \), such that the following inequality holds:
\begin{equation}\label{variational inequality-state constraint}
\begin{aligned}
	&\psi_1 \mathbb{E}\left[G_{1x}(x(T)) x_1(T)\right] + \psi_2 \mathbb{E}\left[G_{0y}(y(0)) y_1(0)\right] \\
	&+ \psi_3 \left( \mathbb{E} \int_0^T \left[g_x x_1 + g_y y_1 + g(u^\varepsilon) - g(u)\right] dt \right.\\
    &\qquad\quad +\ \mathbb{E}[h_x(x(T)) x_1(T)] + \mathbb{E}[\gamma_y(y(0)) y_1(0)] \bigg) \geqslant o(\varepsilon).
\end{aligned}
\end{equation}
where $(x_1(\cdot),y_1(\cdot),z_2(\cdot))$ satisfies (\ref{variational equation}).
\begin{proof}
	Since
	\[
	J_\rho(u_\rho^\varepsilon(\cdot)) - J_\rho(u_\rho(\cdot)) + \sqrt{\rho} \varepsilon \geqslant 0,
	\]
	we have
	\[
	\frac{J_\rho^2(u_\rho^\epsilon(\cdot)) - J_\rho^2(u_\rho(\cdot))}{J_\rho(u_\rho^\epsilon(\cdot)) + J_\rho(u_\rho(\cdot))} + \sqrt{\rho} \varepsilon \geqslant 0.
	\]
	According to Proposition \ref{proposition3.1}, we have
    $$
    \begin{aligned}
	    &J_\rho^2(u_\rho^\epsilon(\cdot))-J_\rho^2(u_\rho(\cdot))\\
        =&\left\{|\mathbb{E}[G_1(x_\rho^\varepsilon(T))]|^2-|\mathbb{E}[G_1(x_\rho(T))]|^2\right\}
         +\left\{|\mathbb{E}[G_0(y_\rho^\varepsilon(0))]|^2-|\mathbb{E}[G_0(y_\rho(0))]|^2\right\}\\
        &+\left\{[J(u_\rho^\varepsilon(\cdot))-J(u(\cdot))+\rho]^2-[J(u_\rho(\cdot))-J(u(\cdot))+\rho]^2\right\}\\
        =&\ 2\mathbb{E}\left[G_1(x_\rho(T))\right]\mathbb{E}\left[G_{1x}(x_\rho(T))x_\rho^1(T)\right]
         +2\mathbb{E}\left[G_0(y_\rho(0))\right]\mathbb{E}\left[G_{0y}(y_\rho(0))y_\rho^1(0)\right]\\
        &+2[J(u_\rho(\cdot))-J(u(\cdot))+\rho] \left\{\mathbb{E}\int_0^T\left[g_{x,\rho}x_\rho^1(t)+g_{y,\rho}y_\rho^1(t)+g(u_\rho^\varepsilon(t))-g(u_\rho(t))\right]dt \right.\\
        &+\mathbb{E}\left[h_x(x_\rho(T))x_\rho^1(T)\right] +\mathbb{E}\left[\gamma_y(y_\rho(0))y_\rho^1(0)\right]\bigg\} +o(\varepsilon).
    \end{aligned}
    $$
	Thus, we obtain
	\begin{equation}\label{variational inequality-state constraints---}
	\begin{aligned}
		\frac{J_\rho^2(u_\rho^\epsilon(\cdot)) - J_\rho^2(u_\rho(\cdot))}{J_\rho(u_\rho^\epsilon(\cdot)) + J_\rho(u_\rho(\cdot))}
		=&\ \psi^\varepsilon_{\rho1} \mathbb{E} \left[G_{1x}(x_\rho(T)) x_\rho^1(T)\right] + \psi^\varepsilon_{\rho2} \mathbb{E}\left[G_{0y}(y_\rho(0)) y_\rho^1(0)\right] \\
		&+ \psi^\varepsilon_{\rho3} \left\{ \mathbb{E} \int_0^T \left[ g_{x,\rho} x_\rho^1(t) + g_{y,\rho} y_\rho^1(t) + g(u_\rho^\varepsilon(t)) - g(u_\rho(t)) \right] dt \right. \\
		& \qquad\quad + \mathbb{E}\left[h_x(x_\rho(T)) x_\rho^1(T)\right] + \mathbb{E}\left[\gamma_y(y_\rho(0)) y_\rho^1(0)\right] \bigg\},
	\end{aligned}
	\end{equation}
	where
	\[
	\psi_{\rho 1}^\varepsilon := \frac{2 \mathbb{E}[G_1(x_\rho(T))]}{J_\rho(u_\rho^\varepsilon(\cdot)) + J_\rho(u_\rho(\cdot))}, \
	\psi_{\rho 2}^\varepsilon := \frac{2 \mathbb{E}[G_0(y_\rho(0))]}{J_\rho(u_\rho^\varepsilon(\cdot)) + J_\rho(u_\rho(\cdot))}, \
	\psi_{\rho 3}^\varepsilon := \frac{2 \left[ J(u_\rho(\cdot)) - J(u(\cdot)) + \rho \right]}{J_\rho(u_\rho^\varepsilon(\cdot)) + J_\rho(u_\rho(\cdot))}.
	\]

	Since
	\[
	\lim\limits_{\varepsilon \rightarrow 0} J_\rho(u^\varepsilon_\rho(\cdot)) = J_\rho(u_\rho(\cdot)),
	\]
	then
	\[
	\lim_{\varepsilon \to 0} \left( |\psi_{\rho 1}^\varepsilon|^2 + |\psi_{\rho 2}^\varepsilon|^2 + |\psi_{\rho 3}^\varepsilon|^2\right) = 1.
	\]
	Therefore, there exists a convergent subsequence (still denoted as) \( (\psi_{\rho1}^\varepsilon, \psi_{\rho2}^\varepsilon, \psi_{\rho3}^\varepsilon) \to (\psi_{\rho1},
    \psi_{\rho2},\\ \psi_{\rho3}) \) as \( \varepsilon \to 0 \), such that
	\[
	|\psi_{\rho 1}|^2 + |\psi_{\rho 2}|^2 + |\psi_{\rho 3}|^2 = 1.
	\]
	Similarly, there exists a convergent subsequence (still denoted as) \( (\psi_{\rho1}, \psi_{\rho2}, \psi_{\rho3}) \to (\psi_{1}, \psi_{2}, \psi_{3}) \) as \( \rho \to 0 \),
    satisfying
	\[
	|\psi_1|^2 + |\psi_2|^2 + |\psi_3|^2 = 1.
	\]
	Since \( u_\rho^\varepsilon(\cdot) \to u^\varepsilon(\cdot) \) and \( u_\rho \to u(\cdot) \) as \( \rho \to 0 \), by the continuous dependence of parameters of the solution to FBSDE with sub-diffusion (Lemma 4.2 in Zhang and Chen \cite{ZC2024-SICON-2}), we have
	\[
	x_\rho^1(\cdot) \to x_1(\cdot), \quad y_\rho^1(\cdot) \to y_1(\cdot), \quad z_\rho^1(\cdot) \to z_1(\cdot),
	\]
	\[
	x_\rho(\cdot) \to x(\cdot), \quad y_\rho(\cdot) \to y(\cdot), \quad z_\rho(\cdot) \to z(\cdot),\quad \text{in }\mathcal{M}_a[0, T].
	\]
	Substituting these results into (\ref{variational inequality-state constraints---}) and letting \( \rho \to 0 \) yields
	\[
	\begin{aligned}
		\frac{J_\rho^2(u_\rho^\epsilon(\cdot)) - J_\rho^2(u_\rho(\cdot))}{J_\rho(u_\rho^\epsilon(\cdot)) + J_\rho(u_\rho(\cdot))}
		= &\ \psi_1 \mathbb{E}[G_{1x}(x(T)) x_1(T)] + \psi_2 \mathbb{E} [G_{0y}(y(0)) y_1(0)] \\
		&+ \psi_3 \left\{ \mathbb{E} \int_0^T \left[ g_x x_1 + g_y y_1(t) + g(u^\varepsilon) - g(u) \right] dt \right. \\
		& \qquad\quad + \mathbb{E}[h_x(x(T)) x_1(T)] + \mathbb{E}[\gamma_y(y(0)) y_1(0)] \bigg\} \geqslant o(\varepsilon).
	\end{aligned}
	\]
	Thus, (\ref{variational inequality-state constraint}) holds.
\end{proof}
\end{mypro}

Next, we introduce the adjoint equation:
\begin{equation}\label{adjoint equation-state constraint}
\left\{
	\begin{aligned}
		dp(t) &= [f_y p(t) - b_y q(t) - \psi_3 g_y] dt - \sigma_y k(t) d{L_{(t-a)+}} - \sigma_z k(t) dB_{L_{(t-a)+}}, \\
		-dq(t) &= [-f_x p(t) + b_x q(t) + \psi_3 g_x] dt + \sigma_x k(t) d{L_{(t-a)+}} - k(t) dB_{L_{(t-a)+}}, \\
		p(0) &= -G_{0y}(y(0)) \psi_2 - \gamma_y y(0) \psi_3, \\
		q(T) &= G_{1x}(x(T)) \psi_1 - \phi_x(x(T)) p(T) + h_x(x(T)) \psi_3.
	\end{aligned}
\right.
\end{equation}
Similarly, by Theorem \ref{th1}, there exists a unique adapted solution \( (p(\cdot), q(\cdot), k(\cdot))\in \mathcal{M}_a[0, T] \) to (\ref{adjoint equation-state constraint}). Define the Hamiltonian:
\[
H(t, x, y, z, v, p, q) := q(t) b(t, x, y, v) - p(t) f(t, x, y, v) + g(t, x, y, v).
\]

Now, we present the main theorem of this section.

\begin{mythm}\label{theorem4.1}
Let \(\ u(\cdot) \)  be an optimal control and\( (x(\cdot), y(\cdot), z(\cdot)) \) the corresponding solution to the FBSDE with sub-diffusion (\ref{controlled forward-backward stochastic system}). Then there exist three constants \( \psi_1 \), \( \psi_2 \), \( \psi_3 \), not all zero, such that \( (p(\cdot), q(\cdot), k(\cdot)) \) is the solution to the adjoint equation (\ref{adjoint equation-state constraint}). Moreover, for any \( v(\cdot) \in \mathcal{U}_a^\prime[0, T] \), the following inequality holds:
\begin{equation}\label{SMP-state constraint}
H(t, x(t), y(t), z(t), v(t), p(t), q(t)) \geqslant H(t, x(t), y(t), z(t), u(t), p(t), q(t)),\ a.e.\ t\in[0,T],\ a.s..
\end{equation}
\begin{proof}
	Applying Ito's formula to \( p(\cdot) y_1(\cdot) + q(\cdot) x_1(\cdot) \) on the interval \( [0, T] \), we have
	\[
	\begin{aligned}
		&p(T) y_1(T) + q(T) x_1(T) - p(0) y_1(0) \\
		&= \int_0^T \Big\{-p \left[ f_x x_1 + f_y y_1 + f(u^\varepsilon) - f(u) \right]\Big\} dt + \int_0^T \left[f_y p - b_y q - \psi_3 g_y\right] y_1 dt \\
		&\quad - \int_0^T \left[k(t) \sigma_y y_1 + k(t) \sigma_z z_1\right] d{L_{(t-a)+}} \\
		&\quad + \int_0^T \left[ b_x x^1 + b_y y^1 + b(u^\varepsilon) - b(u) \right] q dt + \int_0^T \left[f_x p - b_x q - \psi_3 g_x\right] x_1 dt \\
		&\quad + \int_0^T [k \sigma_y y_1 + k \sigma_z z_1] d{L_{(t-a)+}} + \text{martingale}.
	\end{aligned}
	\]
	Substituting into the adjoint equation (\ref{adjoint equation-state constraint}) and taking expectations on both sides, it yields
	\[
	\begin{aligned}
		&\mathbb{E}\left[ y_1(0) G_{0y}(y(0)) \psi_2 + y_1(0) \gamma_y y(0) \psi_3 + x_1(T) G_{1x}(x(T)) \psi_1 + x_1(T) h_x(x(T)) \psi_3 \right] \\
		&= \mathbb{E} \int_0^T \Big\{-p \left[ f(u^\varepsilon) - f(u) \right] + q \left[ b(u^\varepsilon) - b(u) \right] - \psi_3 (g_x x^1 + g_y y^1) \Big\} dt.
	\end{aligned}
	\]
	Using the variational inequality (\ref{variational inequality-state constraint}), we get
	\[
	\mathbb{E} \int_0^T \Big\{ -p \left[ f(u^\varepsilon) - f(u) \right] + q \left[ b(u^\varepsilon) - b(u) \right] + g(u^\varepsilon) - g(u) \Big\} dt \geqslant o(\varepsilon).
	\]
	By the definition of the Hamiltonian, let \( \varepsilon \to 0 \), (\ref{SMP-state constraint}) holds.
\end{proof}
\end{mythm}

In particular, when $G_1=G_2=0, \psi_1=\psi_2=0$ and \( \psi_3 = 1 \), the result of Theorem \ref{theorem4.1} reduces to that of Theorem \ref{theorem3.1}.

\section{Applications to cash management problem in bear markets}

In this section, we apply preceding theoretical results of stochastic maximum principles of FBSDEs driven by sub-diffusion, to model cash management optimization in bear markets. While FBSDEs have been widely applied to cash management optimization in active markets (for example, Bensoussan et al. \cite{BCS2009}), research on their applications to bear markets remains limited. In fact, in bear markets with infrequent trading, traditional stochastic modeling methods cannot accurately represent market dynamics. To address this, we use FBSDEs driven by sub-diffusion to model bear market conditions.

Consider the following FBSDE driven by sub-diffusion:
\begin{equation}\label{example FBSDE}
\left\{
	\begin{aligned}
		dx^v(t) &= \left(-m_1x^v(t)-n_1y^v(t)+c_1v(t)\right)dt-n_2z^v(t)dB_{L_{(t-a)+}},  \\
		-dy^v(t) &= \left(-m_1y^v(t)+m_2x^v(t)+c_2v(t)\right)dt-z^v(t)dB_{L_{(t-a)+}},  \\
		x^v(0)&= x_0,\quad y^v(T)=ax^v(T).
	\end{aligned}
\right.
\end{equation}
In this model, \( x^v(t) \) represents the agent's cash flow at time \( t \), \( v(t) \) represents the agent's investment and withdrawal decisions at time \( t \), \( y^v(t) \) represents the recursive utility of \( v(t) \) at time \( t \), and \( z^v(t) \) represents the volatility of the utility. Due to the significant mutual influence between cash flow and recursive utility, this FBSDE (\ref{example FBSDE}) is fully coupled, and \( x^v(\cdot) \) and \( y^v(\cdot) \) exhibit a negative correlation. 	
The cost functional is
\begin{equation}\label{cost functional-example}
	J(v(\cdot)) := \mathbb{E} \left[\frac{1}{2} \int_0^T (v(t) - l(t))^2 dt - y^v(0) \right],
\end{equation}
where \( l(\cdot) \) is a deterministic bounded function.
	
The objective is to find \( u(\cdot) \in \mathcal{U}_a^\prime[0,T] \) such that
\[
    J(u(\cdot)) = \inf_{v(\cdot) \in\, \mathcal{U}_a^\prime[0,T]} J(v(\cdot)).
\]
The cost functional indicates that the agent aims to maximize recursive utility while minimizing investment risk to achieve a given investment objective.	
	
Here, \( a \geqslant 0, m_i \geqslant 0, n_i \geqslant 0 \, (i=1,2) \). Based on Assumption {\ref{assumption2}} and Theorem \ref{th1}, for any \( v(\cdot) \in \mathcal{U}_a^\prime[0,T] \), equation (\ref{example FBSDE}) admits a unique adapted solution \( (x^v(\cdot),y^v(\cdot)),z^v(\cdot)) \in \mathcal{M}_a[0, T]\).
	
We apply the stochastic maximum principle obtained in Section 3 (Theorem \ref{theorem3.1}), to solve this problem. In this case,
\[
	b(t,x,y,v) = -m_1x - n_1y + c_1v, \quad \sigma(t,x,y,z) = -n_2z, \quad f(t,x,y,v) = -m_1y + m_2x + c_2v,
\]
\[
	\phi(x) = ax, \quad g(t,x,y,v) = \frac{(v - l(t))^2}{2}, \quad \gamma(y) = -y, \quad h(x) = 0.
\]
Thus, the adjoint equation (\ref{adjoint equation}) writes
\begin{equation}\label{adjoint equation-example}
\left\{
	\begin{aligned}
		dp(t) &= [-m_1p(t) + n_1q(t)] dt + n_2k(t) dB_{L_{(t-a)+}}, \\
		-dq(t) &= [-m_2p(t) - m_1q(t)] dt - k(t) B_{L_{(t-a)+}}, \\
		p(0) &= 1, \quad q(T) = -ap(T),
	\end{aligned}
\right.
\end{equation}
and	the Hamiltonian is
\begin{equation}\label{Hamilton-example}
		H(t, x, y, z, v, p, q) = q(-m_1x - n_1y + c_1v) - p(-m_1y + m_2x + c_2v)+ \frac{(v - l(t))^2}{2}.
\end{equation}
By (\ref{SMP}), the optimal control $u(\cdot)$ satisfies
\begin{equation}\label{optimal control-example}
	u(t) = -c_1q(t) + c_2p(t) + l(t),\ a.e.\ t\in[0,T],\ a.s..
\end{equation}

Since Theorem \ref{th1} gives necessary conditions for the optimal control, but not sufficient ones, we cannot directly conclude that \( u(\cdot) \) is really optimal. Next, we verify that \( u(t) \) of (\ref{optimal control-example}) is indeed the optimal strategy. Let \( v(\cdot) \) be any strategy in \( \mathcal{U}_a^\prime[0,T] \), and as before, denote the solution to (\ref{example FBSDE}) associated with \( v(\cdot) \) as \( (x^v(\cdot), y^v(\cdot), z^v(\cdot)) \) and the solution associated with \( u(\cdot) \) as \( (x(\cdot), y(\cdot), z(\cdot)) \). We define
\[
	\tilde{x}(t) := x^v(t) - x(t), \quad \tilde{y}(t) := y^v(t) - y(t), \quad \tilde{z}(t) := z^v(t) - z(t), \quad \tilde{v}(t) := v(t) - u(t).
\]
Applying It\^o's formula to \( \tilde{x}(\cdot)q(\cdot) + \tilde{y}(\cdot)p(\cdot) \), we have
\[
	\begin{aligned}
		\mathbb{E}[\tilde{y}(0)] &= \mathbb{E} \int_0^T \Big\{-p(t)[-m_1\tilde{y}(t)+ m_2\tilde{x}(t) + c_2\tilde{v}(t)] + q(t)[-m_1\tilde{x}(t) - n_1\tilde{y}(t) + c_1\tilde{v}(t)] \\
		&\qquad\qquad + \tilde{y}(t)[-m_1p(t) + n_1q(t)] - \tilde{x}(t)[-m_2p(t) - m_1q(t)] \Big\} dt\\
	    &= \mathbb{E} \int_0^T [-c_2p(t) + c_1q(t)]\tilde{v}(t) dt.
	\end{aligned}
\]
Noticing that
\[
	\mathbb{E}[\tilde{y}(0)] = \mathbb{E} \int_0^T (l(t) - u(t))(u(t) - v(t)) dt.
\]
Thus
\[
\begin{aligned}
	J(v(\cdot))-J(u(\cdot))&=\mathbb{E}\left[\frac{1}{2}\int_0^T \Big[(v(t)-l(t))^2-(u(t)-l(t))^2\Big] dt -\widetilde{y}(0)\right]\\
    &=\mathbb{E}\left[\int_0^T \Big[\frac{(v(t)-l(t))^2-(u(t)-l(t))^2}{2}-(l(t)-u(t))(u(t)-v(t))\Big] dt\right]\\
    &=\mathbb{E}\left[\int_0^T\frac{(v(t)-u(t))^2}{2}{d}t\right] \geqslant 0.
\end{aligned}
\]
Therefore, \( u(\cdot) \) is indeed the optimal strategy.

\section{Concluding remarks}

In this paper, optimal control problems of FBSDEs driven by sub-diffusion are studied. The stochastic maximum principle is obtained, where the diffusion coefficient $\sigma$ is independent of control variable and the control domain may not be convex. Furthermore, the case with initial and terminal constraints state constraints is also addressed, by employing Ekeland's variational principle. Theoretical results are applied to a cash management problem, where the optimal strategy is derived explicitly.

When the diffusion coefficient $\sigma$ depends on the control, but the control domain is required to be a convex set, the stochastic maximum principle has been established in \cite{ZC2024-SICON-2} through convex variation. However, for the case where the diffusion coefficient $\sigma$ is dependent of control and the control domain may not be convex, the corresponding maximum principle still need to be researched. It can be anticipated that its proof would require the introduction of second-order variational equations and two adjoint equations (See \cite{Peng1990}, \cite{Hu2017}, \cite{HJX2018} for standard diffusions' cases). Furthermore, the results can be extended to mean-field type systems, systems with random jumps, partial information, and differential games. For constrained cases, the constraint conditions could also be modified by incorporating integral type constraints. We will consider these topic in the near future.

\end{document}